\documentclass[12pt, a4paper,reqno]{amsart}

\usepackage{a4}
\usepackage{amssymb}
\usepackage{amsmath}
\usepackage{upref}
\usepackage[active]{srcltx}
\usepackage[colorlinks,citecolor=blue,linkcolor=blue]{hyperref}
\usepackage[dvipsnames]{color}
\allowdisplaybreaks[2] 
\definecolor{blau}{rgb}{0.1,0.0,0.9}
\definecolor{purple}{rgb}{0.4,0.0,0.9} 
\definecolor{gruen}{cmyk}{1.0,0.2,0.7,0.07}
\definecolor{mag}{cmyk}{0.0,0.9,0.3,0.0}

\newcounter{komcounter}
\numberwithin{komcounter}{section}

\newcount\minutes \newcount\hours
\hours=\time
\divide\hours 60
\minutes=\hours
\multiply\minutes -60
\advance\minutes \time
\newcommand{\klockan}{\the\hours:{\ifnum\minutes<10 0\fi}\the\minutes}
\newcommand{\tid}{\today\ \klockan}
\newcommand{\prtid}{\smash{\raise 10mm \hbox{\LaTeX ed \tid}}}
\makeatletter
\def\sectionmark#1{} 
\def\subsectionmark#1{}
\newcommand{\sectnr}{\ifnum \c@secnumdepth >\z@
                 \thesection.\hskip 1em\relax \fi}
\def\tableofcontents{\section*{Contents} 
 \@starttoc{toc}}
\makeatother
\makeatletter
\def\@biblabel#1{#1.}
\makeatother
\makeatletter
\let\Thebibliography=\thebibliography
\renewcommand{\thebibliography}[1]{\def\@mkboth##1##2{}\Thebibliography{#1}
\addcontentsline{toc}{section}{References}
\frenchspacing 
\setlength{\@topsep}{0pt}
\setlength{\itemsep}{0pt}%
\setlength{\parskip}{0pt plus 2pt}%
}
\makeatother
\makeatletter
\def\mdots@{\mathinner.\nonscript\!.%
 \ifx\next,.\else\ifx\next;.\else\ifx\next..\else
 \nonscript\!\mathinner.\fi\fi\fi}
\let\ldots\mdots@
\let\cdots\mdots@
\let\dotso\mdots@
\let\dotsb\mdots@
\let\dotsm\mdots@
\let\dotsc\mdots@
\def\vdots{\vbox{\baselineskip2.8\p@ \lineskiplimit\z@
    \kern6\p@\hbox{.}\hbox{.}\hbox{.}\kern3\p@}}
\def\ddots{\mathinner{\mkern1mu\raise8.6\p@\vbox{\kern7\p@\hbox{.}}%
    \raise5.8\p@\hbox{.}\raise3\p@\hbox{.}\mkern1mu}}
\makeatother
\makeatletter
\let\Enumerate=\enumerate
\renewcommand{\enumerate}{\Enumerate%
\setlength{\@topsep}{0pt}
\setlength{\itemsep}{0pt}%
\setlength{\parskip}{0pt plus 1pt}%
\renewcommand{\theenumi}{\textup{(\alph{enumi})}}%
\renewcommand{\labelenumi}{\theenumi}%
}
\let\endEnumerate=\endenumerate
\renewcommand{\endenumerate}{\endEnumerate\unskip}

\makeatother
\makeatletter
\def\@seccntformat#1{\csname the#1\endcsname.\quad}
\makeatother
\makeatletter
\long\def\@makecaption#1#2{%
  \vskip\abovecaptionskip
  \sbox\@tempboxa{ #1. #2}%
  \ifdim \wd\@tempboxa >\hsize
    #1. #2\par
  \else
    \global \@minipagefalse
    \hb@xt@\hsize{\hfil\box\@tempboxa\hfil}%
  \fi
  \vskip\belowcaptionskip}
\makeatother
\RequirePackage{amsthm}
\newtheoremstyle{descriptive}%
{ \medskipamount}          
  {}   
  {\rmfamily} 
  {}          
  {\bfseries} 
  {.}         
  { }         
  {}          
\newtheoremstyle{propositional}%
 {\medskipamount}          
   {\medskipamount}          
  {\itshape}  
  {}          
  {\bfseries} 
  {.}         
  { }         
  {}          
\newtheoremstyle{remarkstyle}%
 {\medskipamount}          
  {\medskipamount}          
  {\rmfamily}  
  {}          
  {\itshape} 
  {.}         
  { }         
  {}          
\theoremstyle{propositional}
\newtheorem{thm}{Theorem}[section]

\newtheorem{theorem}[thm]{Theorem}
\newtheorem{lemma}[thm]{Lemma}

\newtheorem{corollary}[thm]{Corollary}

\theoremstyle{descriptive}

\newtheorem{remark}[thm]{Remark}

\makeatletter
\renewenvironment{proof}[1][\proofname]{\par
  \pushQED{\qed}%
  \normalfont
  \trivlist
  \item[\hskip\labelsep
        \itshape
    #1\@addpunct{.}]\ignorespaces
}{%
  \popQED\endtrivlist\@endpefalse
}
\makeatother
\makeatletter
\newcommand{\AppendicesFromNowOn}
  {\renewcommand \thesection {\@Alph\c@section}  
\newcommand\Appendix{\@startsection {section}{1}{\z@}%
                                   {-3.5ex \@plus -1ex \@minus -.2ex}%
                                   {2.3ex \@plus.2ex}%
                                   {\noindent\normalfont\Large\bfseries Appendix }}
   \setcounter{section}{0}
  }
\makeatother

\def\vint{\mathop{\mathchoice%
          {\setbox0\hbox{$\displaystyle\intop$}\kern 0.22\wd0%
           \vcenter{\hrule width 0.6\wd0}\kern -0.82\wd0}%
          {\setbox0\hbox{$\textstyle\intop$}\kern 0.2\wd0%
           \vcenter{\hrule width 0.6\wd0}\kern -0.8\wd0}%
          {\setbox0\hbox{$\scriptstyle\intop$}\kern 0.2\wd0%
           \vcenter{\hrule width 0.6\wd0}\kern -0.8\wd0}%
          {\setbox0\hbox{$\scriptscriptstyle\intop$}\kern 0.2\wd0%
           \vcenter{\hrule width 0.6\wd0}\kern -0.8\wd0}}%
          \mathopen{}\int}
\def\vintslides{\mathop{\mathchoice%
          {\setbox0\hbox{$\displaystyle\intop$}\kern 0.22\wd0%
           \vcenter{\hrule height 0.04em width 0.6\wd0}\kern -0.82\wd0}%
          {\setbox0\hbox{$\textstyle\intop$}\kern 0.2\wd0%
           \vcenter{\hrule height 0.04em width 0.6\wd0}\kern -0.8\wd0}%
          {\setbox0\hbox{$\scriptstyle\intop$}\kern 0.2\wd0%
           \vcenter{\hrule height 0.04em width 0.6\wd0}\kern -0.8\wd0}%
          {\setbox0\hbox{$\scriptscriptstyle\intop$}\kern 0.2\wd0%
           \vcenter{\hrule height 0.04em width 0.6\wd0}\kern -0.8\wd0}}%
          \mathopen{}\int}
{\catcode`p =12 \catcode`t =12 \gdef\eeaa#1pt{#1}}      
\def\accentadjtext#1{\setbox0\hbox{$#1$}\kern   
                \expandafter\eeaa\the\fontdimen1\textfont1 \ht0 }
\def\accentadjscript#1{\setbox0\hbox{$#1$}\kern 
                \expandafter\eeaa\the\fontdimen1\scriptfont1 \ht0 }
\def\accentadjscriptscript#1{\setbox0\hbox{$#1$}\kern   
                \expandafter\eeaa\the\fontdimen1\scriptscriptfont1 \ht0 }
\def\accentadjtextback#1{\setbox0\hbox{$#1$}\kern       
                -\expandafter\eeaa\the\fontdimen1\textfont1 \ht0 }
\def\accentadjscriptback#1{\setbox0\hbox{$#1$}\kern     
                -\expandafter\eeaa\the\fontdimen1\scriptfont1 \ht0 }
\def\accentadjscriptscriptback#1{\setbox0\hbox{$#1$}\kern 
                -\expandafter\eeaa\the\fontdimen1\scriptscriptfont1 \ht0 }

\newdimen\extrawidth
\def\iintlim#1#2{\setbox0\hbox{$\scriptstyle#1$}%
	\setbox1\hbox{$\scriptstyle#2$}%
	\extrawidth=\wd1 \advance\extrawidth-\wd0
	\ifdim\extrawidth<0pt \extrawidth=0pt\fi%
	\int_{#1\kern\extrawidth \kern .5em}^{#2\kern -\wd1} \kern -.5em%
}
\numberwithin{equation}{section}
\renewcommand{\phi}{\varphi}
\newcommand{\eps}{\varepsilon}

\renewcommand{\l}{\left}
\renewcommand{\r}{\right}

\newcommand{\R}{\mathbb{R}}

\newcommand{\Rn}{\mathbb{R}^d}

\newcommand{\abs}[1]{\left| #1 \right|}

\newcommand{\beq}{\begin{equation}}
\newcommand{\eeq}{\end{equation}}

\title[Parabolic Harnack inequality for non-local diffusion]{On the parabolic Harnack inequality for non-local diffusion equations}
\author[Dier, Kemppainen, Siljander, Zacher]{Dominik Dier, Jukka Kemppainen, Juhana Siljander and Rico Zacher}

\begin{document}



\begin{abstract}
We settle the open question concerning the Harnack inequality for globally positive solutions to non-local in time diffusion equations by constructing a counter-example for dimensions \(d\ge\beta\), where \(\beta\in(0,2]\) is the order of the equation with respect to the spatial variable. The equation can be non-local both in time and in space but for the counter-example it is important that the equation has a fractional time derivative. In this case, the fundamental solution is singular at the origin for all times \(t>0\) in dimensions \(d\ge\beta\). This underlines the markedly different behavior of time-fractional diffusion compared to the purely space-fractional case, where a local Harnack inequality is known. 

The key observation is that the memory strongly affects the estimates. In particular, if the initial data $u_0 \in L^q_{loc}$ for $q$ larger 
than the critical value $\tfrac d\beta$ of the elliptic operator \((-\Delta)^{\beta/2}\), a non-local version of 
the Harnack inequality is still valid as we show.

We also observe the {\em critical dimension phenomenon} already known from other contexts: the diffusion behavior is substantially different in higher dimensions than $d=1$ provided \(\beta>1\), since we prove that the local Harnack inequality holds  if \(d<\beta\).
\end{abstract}

\maketitle

\section{Introduction}

We construct a counter-example for the local Harnack inequality for globally positive solutions to the time-fractional heat equation
\begin{equation}\label{equation}
\partial_t^\alpha(u-u_0) +(- \Delta)^{\beta/2} u = 0, \quad 0 < \alpha <1,\quad 0<\beta\le 2,
\end{equation}
in $(0,T) \times \Rn$ for $d \ge \beta$. We also show that the result is 
still valid in one dimension provided \(\beta>1\), at least for global solutions. Instead of the local formulation, a non-local version of 
the Harnack inequality  is true also in higher dimensions. Indeed, the initial data plays a role, as 
the result remains true if the superharmonic function induced by the nonnegative initial 
data satisfies the Harnack principle. Thus, the memory has a substantial effect on 
the properties of the solutions, a phenomenon not seen in the case of the 
space-fractional diffusion equation
\begin{equation}\label{space_fractional}
\partial_t u + (-\Delta)^{\beta/2} u =0.
\end{equation}

In particular, if $u_0 \in L^q_{loc}$ for  $q > \tfrac d\beta$, we prove a version of the 
classical local Harnack principle with the constant depending on $u_0$. We also observe the {\em critical dimension 
phenomenon}, which is not seen in the case of  equation~\eqref{space_fractional}, but is 
known for time-fractional diffusion from~\cite{KempSiljVergZach16, KempSiljZach16}. These 
results underline the fact that equation~\eqref{equation} has a substantially different 
behavior from that of~\eqref{space_fractional}. The phenomena we see, as well as the 
methods used, are often completely different from the ones for 
equation~\eqref{space_fractional}.

In the purely space-fractional case the Harnack inequality is known and was first proven by heat kernel estimates by Bass and Levin in~\cite{BassLevi02} as well as Chen and Kumagai in~\cite{ChenKuma03}. Very recently a new proof for viscosity solutions was given by Chang-Lara and D\'avila in~\cite{ChLaDavi16}. In the variational formulation the question has been studied for kernels vanishing outside the local diagonal in~\cite{BarlBassChenKass09} and for viscosity solutions when the Harnack constant is allowed to depend on the non-local tail in~\cite{KimLee13}. For elliptic theory we refer to the work of Caffarelli and Silvestre in~\cite{CaffSilv07, CaffSilv09} and Kassmann in~\cite{Kass11} and the references given therein. See also~\cite{BonfVazq13, CaffSilv12}.

The study of equation~\eqref{equation} is firmly rooted in applications. The 
time-fractional diffusion equation is closely related to a so-called continuous time random 
walk (CTRW) model for particle diffusion and it has become one of the standard physics 
approaches to model anomalous diffusion 
processes~\cite{HenrLangWear06,MeerSiko10,MetzKlaf00,MeerBensScheBeck02}. For a 
detailed derivation of this equation from physics principles and for further applications 
of such models we refer to the expository review article of Metzler and Klafter 
in~\cite{MetzKlaf00}. For the connection of~\eqref{equation} to the related 
stochastistic process we refer to monograph~\cite{MeerSiko10} and the extensive bibliography 
therein. Moreover, equations of the form (\ref{equation}) and nonlinear variants of them 
appear also in modelling
of dynamic processes in materials with {\em memory}. Examples are given by the theory of heat
conduction with memory, see~\cite{Prus13} and the references therein, and the
diffusion of fluids in porous media with memory~\cite{Capu99, Jaku02}.

We point out that the non-local in time term in (\ref{equation}), with $\partial_t^\alpha$ being the Riemann-Liouville
fractional derivation operator, coincides (for sufficiently smooth $u$) with the Caputo fractional derivative of $u$,
see (\ref{caputo}) below. The formulation with the Riemann-Liouville fractional 
derivative has the
advantage that a priori less regularity is required on $u$ to define the nonlocal operator. In particular, our formulation is exactly the one which naturally arises from physics applications, see for instance~\cite[equation (40)]{MetzKlaf00}.

\section{Preliminaries and main results} \label{section:preliminaries}

\subsection{Notations and definitions}

Let us first fix some notations. The Riemann-Liouville fractional integral of order \(\alpha \ge 0\) is defined for $\alpha = 0$ as \(J^0:=I\), where \(I\) denotes the identity operator, and for $\alpha >0$ as
\[
J^\alpha 
f(t)=\frac{1}{\Gamma(\alpha)}\int_0^t(t-\tau)^{\alpha-1}f(\tau)d\tau = (g_\alpha * f)(t),
\]
where 
\[
g_\alpha(t)=\frac{t^{\alpha-1}}{\Gamma(\alpha)}
\]
is the Riemann-Liouville kernel and $*$ denotes the convolution in time on the positive half-line. We denote the 
convolution in space by $\star$. 

The Riemann-Liouville fractional derivative of order $0<\alpha  < 1$ is 
defined by
\[
\partial_t^{\alpha}f(t)=\frac{d}{dt} J^{1-\alpha}f(t)=\frac{d}{dt}\frac{1}{\Gamma(1-\alpha)}
\int_0^ { t } (t-\tau)^ { -\alpha } f(\tau)d\tau.
\]
Observe that for sufficiently smooth $f$ and $\alpha\in (0,1)$
\begin{equation} \label{caputo}
\partial_t^{\alpha}(f-f(0))(t)=(J^{1-\alpha} f')(t)=:{}^cD_t^\alpha f(t),
\end{equation}
the so-called Caputo fractional derivative of $f$.
In case $\alpha=1$, we have the standard time derivative. 

Let 
\[
\widehat{u}(\xi)=\mathcal{F}(u)(\xi)=(2\pi)^{-d/2}\int_{\Rn}\mathrm{e}^{
-ix\cdot\xi}f(x)dx
\] 
and
\[
\mathcal{F}^{-1}(u)(\xi):=\mathcal{F}(u)(-\xi)
\]
denote the Fourier and inverse Fourier transform of $u$, respectively. 

%
%
%
%
%
%

\subsection{The fundamental solution}
We recall that under suitable conditions on the initial data the classical solution of the Cauchy problem
\begin{equation}\label{cauchy_problem}
\left\{\begin{array}{r@{\,\,=\,\,}l}
\partial_t^\alpha(u-u_0) +(-\Delta)^{\beta/2} u & 0 \\
u(0,x) & u_0(x)
\end{array}\right.
\end{equation}
is given by the representation formula
\begin{equation}\label{solution_repr}
u(t,x) = \int_{\Rn} Z(t, x-y ) u_0 (y) \, dy
\end{equation}
where $Z$ is the fundamental solution of the problem. We call the function \(u\) defined 
by~\eqref{solution_repr} the {\em mild solution} of the Cauchy 
problem~\eqref{cauchy_problem} whenever the integral in~\eqref{solution_repr} is well 
defined. We remark that under appropriate regularity conditions on the data, existence and 
uniqueness of strong $L^p$-solutions of the Cauchy problem~\eqref{cauchy_problem} follows 
from the results in~\cite{Zach05}, which are formulated in the framework
of abstract parabolic Volterra
equations, see also the monograph~\cite{Prus13}.

The fundamental solution \(Z\) can be expressed in terms of Fox $H$-functions which are quite a general class of special functions. Without going to excessive details concerning the theory of Fox 
$H$-functions, for which  we refer the reader to~\cite{KilbSaig04}, we just give a very brief review, see also~\cite{KempSiljZach16}. 

To simplify the notation we introduce
\[
(a_i,\alpha_i)_{1,p}:=((a_1,\alpha_1),(a_{2},\alpha_{2}),\dots,(a_p,
\alpha_p))
\]
for the set of parameters appearing in the definition of Fox $H$-functions, which are defined via Mellin-Barnes type
integrals as
\[
H^{mn}_{pq} (z):= H^{mn}_{pq}\big[ z \big| \begin{smallmatrix}
                           (a_i,\alpha_i)_{1,p}\\
(b_j,\beta_j)_{1,q}                 
                                           \end{smallmatrix}
\big] =\frac{1}{2\pi i}\int_{\mathcal{L}}\mathcal{H}^{mn}_{pq}(s)z^{-s} 
ds,
\]
where 
\[
\mathcal{H}^{mn}_{pq}(s)=\frac{\prod_{j=1}^m \Gamma(b_j+\beta_j s)\prod_{i=1}^n 
\Gamma(1-a_i-\alpha_i s)}{\prod_{i=n+1}^p\Gamma(a_i+\alpha_i 
s)\prod_{j=m+1}^q\Gamma(1-b_j-\beta_j s)}
\]
is the Mellin transform of \(H_{pq}^{mn}\) and 
\(\mathcal{L}\) is an infinite contour in the complex plane which separates 
the poles
\[
b_{jl}=\frac{-b_j-l}{\beta_j}\quad (j=1,\dots,m;\, l=0,1,2,\dots)
\]
of the Gamma function \(\Gamma(b_j+\beta_j s)\) to the left of \(\mathcal{L}\) 
and the poles
\[
a_{ik}=\frac{1-a_i+k}{\alpha_i} \quad (i=1,\dots,n;\, k=0,1,2,\dots)
\]
of $\Gamma(1-a_i-\alpha_i s)$ to the right of \(\mathcal{L}\).

The fundamental solution satisfying the Cauchy problem~\eqref{cauchy_problem} with the 
initial datum \(u_0\) given by the Dirac delta distribution, \(u_0(x)=\delta(x)\), can be 
represented in terms of a Fox $H$-function as
\begin{equation}\label{fund_sol_repr}
Z(t,x)=\pi^{-d/2}|x|^{-d}H_{23}^{21}\bigl[ 2^{-\beta} 
t^{-\alpha}|x|^{\beta} \big| \begin{smallmatrix} (1,1), 
&(1, \alpha)\\ 
(\frac{d}{2},\frac{\beta}{2}),& (1,1 ), &(1, \frac{\beta}{2})
\end{smallmatrix} \big].
\end{equation}

Our analysis requires the sharp asymptotics of the fundamental solution, which is collected into the following lemma. For the proof we refer to~\cite[Lemma 3.3]{KempSiljZach16}. See also~\cite[equation (3.7)]{EideKoch04} 
and~\cite[Theorem 1.10]{KilbSaig04}.

\begin{lemma}\label{fund_sol_asympt}

Let $\alpha\in(0,1]$, $\beta\in(0,2]$ and $d \in \mathbb{Z}_+$. Denote \(R:=2^{-\beta}|x|^\beta 
t^{-\alpha}\). The function $Z$ has the following asymptotic behavior:
\begin{itemize}

\item[(i)] If \(R\leq 1\), then
\[
Z(t,x)\sim \begin{cases}
             t^{-\alpha d/\beta}, &\textrm{if $\alpha=1$, or $\beta > d$ and 
$0<\alpha<1$},\\
t^{-\alpha}(|\log(|x|^\beta t^{-\alpha})|+1), &\textrm{if $\beta=d$ and  
$0<\alpha <1$},\\
t^{-\alpha}|x|^{-d+\beta} &\textrm{if $0<\beta < d$ and $0<\alpha<1$}.
            \end{cases}
\]

\item[(ii)] If \(R\ge 1\), then 
\[
Z(t,x)\sim 
\begin{cases}
t^{\alpha}|x|^{-d-\beta}, &\textrm{if \(\beta<2\)},\\
t^{-\frac{\alpha d}2}R^{\frac{d(\alpha-1)}{2(2-\alpha)}}e^{-\sigma 
R^{\frac1{2-\alpha}}}, &\textrm{if \(\beta=2\)},
\end{cases}
\]
where \(\sigma=(2-\alpha)\left(\alpha^\alpha/4\right)^{\frac{1}{2-\alpha}}\).
\end{itemize}

\end{lemma}
In Lemma~\ref{fund_sol_asympt} and in the sequel we use the notation
\[
a\sim b\quad\textrm{if there exists a constant $c>0$ such that}\quad \frac{1}{c} a\le b\le c a.
\]
Lemma~\ref{fund_sol_asympt} guarantees that we have not only upper bounds but also lower bounds for the fundamental solution \(Z\). The upper and the lower bounds are the same up to some multiplicative constants. In our analysis we are not interested in the exact constants. Therefore we give our estimates in a form \(a\le C b\), or even more briefly in a form \(a\lesssim b\).

Although the Fox $H$-functions are rather complicated objects, the Fourier transform with 
respect to the spatial variable of the fundamental solution \(Z(t,\cdot)\) takes a much simpler form. In fact,
\begin{equation}\label{fund_fourier_trans}
\widehat{Z}(t,\xi)=(2\pi)^{-d/2} E_{\alpha}(-\abs{\xi}^\beta t^{\alpha}),
\end{equation}
where $E_{\alpha}$ is the Mittag-Leffler function, which is defined by
\[
E_{\alpha}(z)=\sum_{k=0}^\infty \frac{z^k}{\Gamma(1+\alpha k)},\quad z\in \mathbb{C}.
\]
It is known that for \(0<\alpha<1\) the Mittag-Leffler function behaves as
\begin{equation}\label{ML2_estimate}
E_{\alpha}(-x) \sim \frac{1}{1+x}, \quad x\in\mathbb{R}_+.
\end{equation}
For details 
we refer to~\cite[Chapter 18]{Erdelyi53}.

\subsection{Main results}

The question about a full, local Harnack inequality
\begin{equation}\label{local_harnack}
\sup_{Q^-} u \le C \inf_{Q^+} u,
\end{equation}
for nonnegative solutions of equations of type~\eqref{equation} and for some time-lagged properly scaled cylinders $Q^-$ and $Q^+$ has been a somewhat 
longstanding open question. The 
constant $C>0$ in the classical Harnack estimate~\eqref{local_harnack} should neither depend on the 
initial datum $u_0$ and the solution $u$ nor on the scaling. It turns out that for dimensions \(d\ge \beta\) this is not the case 
for the solution of~\eqref{equation}. Our first result is the construction of a 
counter-example for such a Harnack principle. We have the following theorem.

\begin{theorem}\label{counter_example}
Let $\alpha\in(0,1)$, $\beta\in(0,2]$ and $d \ge \beta$. There exists a sequence $(u_n)_{n=1}^\infty$ of smooth, positive 
solutions to equation~\eqref{equation} in $(0,\infty)\times \Rn $ (with $u_n|_{t=0}$ instead of $u_0$) such that the ratio
\begin{equation*}
\frac{u_n(t_1, 0)}{u_n(t_2, x_0)} \to \infty \quad \text{as } n \to \infty,
\end{equation*}
for each pair of time levels $t_1, t_2 >0$ and for each $x_0\neq 0$. This
contradicts the local Harnack inequality~\eqref{local_harnack}.

\end{theorem}

The proof of Theorem~\ref{counter_example} is based on the detailed asymptotics of the fundamental solution \(Z\). We also provide another counter-example, which uses the Fourier transform. 

\begin{theorem}\label{counter_example_boundedness}
Let $\alpha\in(0,1)$ and $\beta\in(0,2]$. There exists a sequence \((u_n)_{n=1}^\infty\) of smooth, positive solutions to equation~~\eqref{equation} in 
$(0,\infty)\times \Rn $ (with $u_n|_{t=0}$ instead of $u_0$)
such that for all \(t_1,t_2>0\) there holds
\[
\frac{u_n(t_1,0)}{\|u_n(t_2,\cdot)\|_{L^p(\Rn)}}\to\infty\quad\text{as } n\to\infty,
\]
if
\begin{itemize}
\item[(i)] \(d>\beta\) and \(1\le p<\frac{d}{\beta}\), or
\item[(ii)] \(d=\beta\) and \(p=1\).
\end{itemize}
\end{theorem} 

In the classical local case the proof of the Harnack inequality consists of the following steps:
\begin{itemize}
\item[(a)] boundedness of subsolutions ('\(L^\infty-L^p\)-- estimate'),
\item[(b)] the weak Harnack inequality for supersolutions ('\(L^p-L^{-\infty}\)-- estimate').
\end{itemize}

In~\cite{Zach12} the weak Harnack inequality is proved in the case \(\beta=2\) for all spatial dimensions \(d\)
and all positive $p<\frac{2+d\alpha}{2+d\alpha-2\alpha}$, with the upper bound for $p$ being optimal.
In~\cite{JiaPengYang17} the weak Harnack inequality is proved in the case \(\frac{d}{2}<\beta<2\) for the spatial dimensions \(d=2,3\). Theorem~\ref{counter_example_boundedness} reveals that the boundedness estimate not only of subsolutions but also
of solutions fails in the cases
\begin{itemize}
\item[(1)] \(1<\beta\le 2\) and \(d\ge 2\),
\item[(2)] \(0<\beta\le 1\) and \(d\ge 1\).
\end{itemize}
The boundedness of solutions is an open problem in the case \(1<\beta\le 2\) and \(d=1\). Hence for subdiffusion there are cases in which the weak Harnack inequality is true but the boundedness of solutions fails. Apart from the cases \(\beta=2\) and \(d\ge 1\) or \(\frac{d}{2}<\beta<2\) and \(d=2,3\) even the weak Harnack inequality is open. 

Concerning the Harnack inequality for non-local diffusion equations, there are some positive results. In~\cite{ChLaDavi16} the Harnack inequality is proved for a generalization of~\eqref{space_fractional} in the case \(1\le\beta <2\). For the non-local in time equations, the only positive result up to now is the Harnack inequality for the Riemann-Liouville fractional derivation operator proved in~\cite{Zach11}, which can be regarded as the case '\(d=0\)'. The case \(0<\alpha<1\) in dimensions \(d\ge 1\) has been an open problem until now. We show in Theorem~\ref{harnack_local_in_time} that the Harnack inequality remains to hold for~\eqref{equation} for \(0<\alpha<1\) in the one-dimensional case provided \(\beta>1\). Having positive results, one could think that the Harnack inequality could hold also for the whole range of parameters \(\alpha\) and \(\beta\), but this is not the case as Theorem~\ref{counter_example} shows in dimensions \(d\ge\beta\).

The local Harnack inequality being settled for dimensions $d \ge \beta$, it is a natural 
question, what is a proper replacement of this result. Unlike for most non-local 
equations, the non-local memory term in the time-fractional diffusion equation has a 
substantial effect on the behavior of solutions. Indeed, large enough historical events 
affect the solution indefinitely for all positive times. From the construction of 
Theorem~\ref{counter_example} it is evident that if the initial data is pathological 
enough, then the solution will never satisfy a local Harnack inequality. If the initial 
data, however, is good enough, then a Harnack principle still holds.

Moreover, the theory is also very different depending on the dimension. In one dimension, we are still able to prove the classical Harnack inequality for global solutions, but in higher dimensions the counter-example prevails. This corresponds to the {\em critical dimension phenomenon} observed in the case of the decay estimates for the equation~\eqref{equation} in~\cite{KempSiljVergZach16}. See also~\cite{KempSiljZach16}.

For dimensions $d \ge\beta$ the equation~\eqref{equation} preserves the  
properties of the initial data in the sense that if the superharmonic function induced by 
the nonnegative initial data satisfies a Harnack principle, also the solution of the 
Cauchy problem 
will have a similar behavior. For making this rigorous, we consider the 
potential
\[
u\mapsto G\star u
\]
with the kernel
\begin{align}\label{riesz_kernel}
G(t,x) = 
\frac{1}{c(d,\beta)}\begin{cases}
|x|^{\beta-d}, \quad &\text{for } d > \beta \\
\log|x|, \quad &\text{for } d =\beta,
\end{cases}
\end{align}
where \(c(d,\beta)\) is a normalization constant. Note that for \(\beta=2\) the function \(G\) is nothing but the Newtonian kernel and the corresponding potential is the Newtonian potential. Now $G 
\star u_0$ induces a superharmonic function in $\Rn$, that is, setting $v=G\star u_0$ for a nonnegative and sufficiently
smooth initial data $u_0$ we obtain that $(-\Delta)^{\beta/2}v\ge 0$ in $\Rn$.

We have the following theorem. Here by $\chi_A(x)$ we denote the characteristic function of the set $A\subset \Rn$.

\begin{theorem}\label{main_theorem}  Let $\alpha\in(0,1)$, $\beta \in(0,2]$ and $r >0$.  
Suppose $u$ is a mild solution of the Cauchy problem~\eqref{cauchy_problem} in 
$[0,T]\times \Rn$ with the nonnegative and sufficiently smooth initial data $u_0\not\equiv0$.  Suppose further that $t_1$ and $t_2$ satisfy condition
\begin{equation}\label{time_cond}
(2r)^{\beta/\alpha} \le  t_1 < t_2 \le t_1+(2r)^{\beta/\alpha} \le T.
\end{equation}
Then 
there exists a constant $C = C(d, \alpha,\beta ,t_2/t_1) >0$ such that for all 
$x_1, x_2 \in B_{r}(0)$ there holds
\begin{equation*}
u(t_1,x_1) \le C \left[1+\frac{[G \star (u_0\chi_{B_{r_1}(x_1)})](x_1)}{[G \star 
(u_0\chi_{B_{r_2}(x_2)})](x_2)}\right] u(t_2,x_2), \quad \text{if }  d>\beta,
\end{equation*}
and, if \(d=\beta\),
\begin{equation*}
u(t_1,x_1)\le C\left[1+\frac{\big[(c+\alpha \log t_1+\beta c(d,\beta)G) \star  
\big(\chi_{B_{r_1}(x_1)}u_0\big)\big](x_1)}{\big[(c+\alpha\log t_2+\beta c(d,\beta)G) \star 
\big(\chi_{B_{r_2}(x_2)}u_0\big)\big](x_2)}
\right] u(t_2, x_2),
\end{equation*}
where \(c=1+\beta\log 2\), \(c(d,\beta)\) is the constant in~\eqref{riesz_kernel} and $r_i =2 t_i^{\alpha/\beta}, i=1,2$. Moreover, the constant \(C\) blows up, as \(t_2\to t_1\), when \(\beta=2\).
\end{theorem}

\begin{remark}\label{time_lag_remark}
Observe that by elliptic regularity theory, if $u_0 \in L^q_{loc}(\Rn)$ with $q > 
\tfrac d\beta$, then $G \star (u_0\chi_{B_{r_1}(x_1)})$ is locally bounded and continuous. In this case the convolution factor above is always bounded 
as long as $B_{r_1}(x_1) \subset B_{r_2}(x_2)$ which is true for all $(t_1, x_1), (t_2, 
x_2)$ satisfying
\begin{equation*}
t_2^{\alpha/\beta} \ge t_1^{\alpha/\beta} + \frac{|x_1-x_2|}{2}
\end{equation*}
and~\eqref{time_cond}. Therefore, the Harnack inequality is always satisfied for such an initial data after a proper time-lag corresponding to the scaling of the equation. 
\end{remark}

We would also like to point out that the proof of the above Theorem is rather sharp consisting almost exclusively of identities rather than estimates. Indeed, the crucial estimates rely on the asymptotic behavior of the fundamental solution and these estimates provide sharp behavior not only as upper bounds, but also as lower bounds.

In the above formulation the effect of the initial data is readily observable. The 
superharmonic function induced by the initial data fully determines the Harnack constant. 
Moreover, it is easy to see that if this superharmonic function satisfies a proper Harnack 
inequality, then this is true also for the solution of the evolution equation. In 
particular, if one is willing to make some more additional assumptions on the initial 
data, we are able to obtain a more classical formulation of the Harnack inequality in the 
form of the following corollary. 


\begin{corollary}\label{main_corollary}
Let the assumptions of Theorem~\ref{main_theorem} hold. Suppose also that the initial data satisfies the Harnack inequality 
\begin{equation}\label{initial_harnack}
\sup_{B_{r_1}(x_1)}u_0 \le H_0 \inf_{B_{r_2}(x_2 )} u_0,
\end{equation}
for $x_1, x_2 \in B_r(0)$ and for some constant $H_0$. Then there exists a constant $C = 
C(d, \alpha,\beta ,H_0, t_2/t_1) >0$ such that there holds
\[
u(t_1,x_1) \le C u(t_2,x_2).
\]
\end{corollary}

In the case \(\beta>d\) the situation is simpler as the counter-example we build in Theorem~\ref{counter_example} is not valid anymore. In this case the Harnack inequality~\eqref{standard_Harnack} takes the standard form. 
We have the following theorem. We formulate the result in a form, which covers also the space-fractional case~\eqref{space_fractional}, since $\alpha=1$ is included in Lemma~\ref{fund_sol_asympt}.

\begin{theorem}\label{harnack_local_in_time}
Let $\alpha\in(0,1],\beta\in(0,2),r>0$ and \(x_1,x_2\in B_r(0)\). Assume that either 
\begin{itemize}
\item[(a)] $\alpha\neq 1$ and $1=d<\beta$, or
\item[(b)] $\alpha=1$, $\beta\neq 2$ and $d\in\mathbb{Z}_+$.
\end{itemize}
Suppose $u$ is a mild solution of the Cauchy 
problem~\eqref{cauchy_problem} in $[0,T]\times \Rn$. 
Then for all \(t_1,t_2\) satisfying
\[
(2r)^{\beta/\alpha} \le t_1 < t_2  \le t_1+(2r)^{\beta/\alpha} \le T.
\]
there exists a constant $C=C(t_2/t_1,\alpha,\beta)>0$ such that
\begin{equation}\label{standard_Harnack}
u(t_1,x_1)\le C u(t_2,x_2),
\end{equation}

Moreover, the constant blows up, as \(t_2\to t_1\), when \(\beta=2\).
\end{theorem}

\begin{remark}

\begin{itemize}

\item[(i)] The proofs of Theorems~\ref{main_theorem} and~\ref{harnack_local_in_time} reveal that the time-lag is needed only for handling the exponential term, which is missing in the case \(\beta\neq 2\). In particular, the time lag is not needed in the case \(\alpha=1\) and \(\beta\neq 2\), which was observed also in~\cite{BonfSireVazq17}. Heuristically this means that the diffusion is so fast that  a lot of heat is diffused far away from the source at all times \(t>0\). In the probabilistic framework this is a consequence of a fat tail at infinity of the probability distribution \(Z(t,\cdot)\).

\item[(ii)] In the fully non-local case \(\alpha\neq 1\) and \(\beta\neq 2\) the diffusion is at the same time slow and fast. It is slow, since the fundamental solution has singularity at \(x=0\) for all times \(t>0\). On the other hand, it is fast since the fundamental solution \(Z(t,\cdot)\) has a fat tail at infinity.
\end{itemize}
\end{remark}


\section{The counter-examples}

We begin with the proof of Theorem~\ref{counter_example}. The proof relies on the 
properties of the fundamental solution. Indeed, the fundamental solution is not a smooth 
function and is merely superparabolic instead of being a proper solution of the problem. 
This resembles the elliptic Newtonian potential and the case of some nonlinear parabolic equations, such as the porous medium 
equation (and the corresponding Barenblatt solution), which have been used to model 
similar phenomena as equation~\eqref{equation}. Equation~\eqref{equation} can be viewed as interpolation between the parabolic and elliptic cases. 

We need the sharp behavior 
of 
the fundamental solution given in Lemma~\ref{fund_sol_asympt}. It seems that in the literature only upper bounds (see, e.g.~\cite[Proposition 1]{EideKoch04} and~\cite{KimLim15}) are available 
although the sharp asymptotics of the Fox $H$-functions is known.
For the convenience of the reader we provide here an easy argument that the fundamental solution is 
actually strictly positive on the whole space \(\mathbb{R}^d\) for all times \(t>0\). It 
is well-known that the fundamental solution is a probability density and hence 
nonnegative everywhere for all \(t>0\) but it does not exclude the possibility of the fundamental solution having zeros at some points.

We first prove the following auxiliary result. The proof is 
based on the Mellin transform of the Fox $H$-function appearing in the representation of 
the fundamental solution and the calculus of residues. For details we refer 
to~\cite{KempSiljZach16,KilbSaig04}.


\begin{lemma}\label{lemma_Z_decreasing}
Let $\alpha\in(0,1]$ and $\beta\in(0,2]$. The fundamental solution \(Z(t,\cdot)\) given by~\eqref{fund_sol_repr} is radially decreasing.
\end{lemma}
 
\begin{proof}
We proceed similarly as in the proof of~\cite[Lemma 4.7]{KempSiljZach16}. 

The fundamental solution given by~\eqref{fund_sol_repr} is clearly radial. With a slight 
abuse of notation we denote \(Z(t,x)=Z(t,r)\), where \(r=|x|\). Using the differentiation 
rule~\cite[Lemma 2.14]{KempSiljZach16} of the Fox $H$-functions we have
\begin{equation}\label{Fox_derivative_r}
\frac{\partial}{\partial r} Z(t, r)= \pi^{-d/2}r^{-d-1}[\beta 
H_{34}^{22}(2^{-\beta}t^{-\alpha}r^\beta)-dH_{23}^{21}(2^{-\beta}t^{-\alpha}r^\beta)],
\end{equation}
where we have noted that \(Z\) is a radial function and denoted \(r=|x|\), and we have omitted the parameters appearing in the H-functions.

We may combine the linear combination
\begin{equation}\label{Fox_linear_combination}
\beta H^{22}_{34}(z)-d H^{21}_{23}(z)
\end{equation}
as follows. Let \(\mathcal{H}^{21}_{23}\) and \(\mathcal{H}^{22}_{34}\) be the Mellin 
transforms of the functions \(H^{21}_{23}\) and \(H^{22}_{34}\) appearing 
in~\eqref{Fox_derivative_r}. Using the definition of the Fox 
$H$-function~\cite[Formulae (2.10) and (2.11)]{KempSiljZach16} we have
\[
\beta \mathcal{H}_{34}^{22}(s)-d\mathcal{H}_{23}^{21}(s)=\beta \frac{\Gamma(\frac{d}{2}+\frac{\beta}{2}s)\Gamma(1+s)\Gamma(1-s)}{\Gamma(1+\alpha 
s)\Gamma(-\frac{\beta}{2}s)} 
-d\frac{\Gamma(\frac{d}{2}+\frac{\beta}{2}s)\Gamma(1+s)\Gamma(-s)}{\Gamma(1+\alpha 
s)\Gamma(-\frac{\beta}{2}s)}.
\]
Using the property \(\Gamma(z+1)=z\Gamma(z)\) of the Gamma function we may proceed as
\[
\begin{split}
\beta \mathcal{H}_{34}^{22}(s)-d\mathcal{H}_{23}^{21}(s)&=-2\l(\frac{d}{2}+\frac{\beta}{2}s\r)\frac{\Gamma(\frac{d}{2}+\frac{\beta}{2}
s)\Gamma(1+s)\Gamma(-s)}{\Gamma(1+\alpha s)\Gamma(-\frac{\beta}{2}s)} \\
&=-2\frac{\Gamma(\frac{d+2}{2}+\frac{\beta}{2}s)\Gamma(1+s)\Gamma(-s)}{\Gamma(1+\alpha 
s)\Gamma(-\frac{\beta}{2}s)} \\
&=-2\mathcal{H}_{23}^{21}\big[ s \big| \begin{smallmatrix}(1,1), 
&(1, \alpha)\\ 
(\frac{d+2}{2},\frac{\beta}{2}),& (1,1 ), &(1, \frac{\beta}{2})&  \end{smallmatrix} \big],
\end{split}
\]
which implies
\[
\frac{\partial}{\partial r} Z(t, r)= -2\pi^{-d/2}r^{-d-1}H_{23}^{21}\big[ 
2^{-\beta} 
t^{-\alpha}r^{\beta} \big| \begin{smallmatrix}(1,1), 
&(1, \alpha)\\ 
(\frac{d+2}{2},\frac{\beta}{2}),& (1,1 ), &(1, \frac{\beta}{2})&  \end{smallmatrix} \big].
\]
which is nothing but
\[
-2\pi r Z_{d+2}(t,r),
\]
where the subscript \(d+2\) means that \(Z_{d+2}\) is the fundamental solution 
\(Z(t,\cdot)\) in dimensions \(d+2\ge 3\). Since \(Z(t,\cdot)\) is a probability density, 
the result follows.
\end{proof}

\begin{corollary}\label{corol_no_pos_zero}
The fundamental solution \(Z(t,r)\) has no positive zeros \(r\in \R_+\).
\end{corollary}

\begin{proof}
Since \(r\mapsto Z(t,r)\) is decreasing and nonnegative in \(\R_+\), it cannot have 
zeros on the positive real axis. Otherwise \(Z(t,\cdot)\) would have compact support, 
which would imply that the analytic function \(H^{21}_{23}(z)\), \(z\neq 0\), would have 
zeros in a continuum, which is impossible.
\end{proof}

Having Lemma~\ref{lemma_Z_decreasing} and Corollary~\ref{corol_no_pos_zero} 
in our hands, we can use the sharp behavior of the fundamental solution given by the 
asymptotics formulated in terms of the similarity variable 
\begin{equation*}
R= \frac{|x|^\beta}{2^\beta t^\alpha}.
\end{equation*}
Note that also the classical heat kernel contains a function depending on \(R\) with \(\alpha=1\) and \(\beta=2\).

\begin{proof}[Proof of Theorem~\ref{counter_example}]
First of all, we note that by scaling it is enough to consider \(x_0\in\partial B(0,1)\). Let $\psi(x)$ be the standard mollifier 
\begin{equation*}
\psi(x) := 
\begin{cases}
c_d e^{-\frac{1}{1-|x|^2}}, \quad \text{for } |x| <1 \\
0, \quad \text{for } |x| \ge1,
\end{cases}
\end{equation*}
where $c_d$ is a dimension-dependent normalization constant to guarantee the integral of  
$\psi$ to equal one. We choose $u_0^\eps(x) = \eps^{-d}\psi(\tfrac x\eps)$ as the initial 
data for our problem. We will find a lower bound for the left hand side 
of~\eqref{local_harnack}, which will tend to infinity as \(\eps\to 0\), and an upper 
bound of the right hand side of~\eqref{local_harnack}, which stays bounded uniformly in 
\(\eps\). We start with the lower bound, where we consider the cases \(d> \beta\) and 
\(d=\beta\) separately, since the asymptotics of the fundamental solution is different in these cases. 

\noindent {\bf The lower bound. Case: $d> \beta$.} Let \(\psi\) and \(u_0^\epsilon\) be as above and let \(t>0\) be fixed. Then 
\[
u^\eps(t,0)=\int_{\mathbb{R}^d} Z(t,y)u_0^\eps(y) dy=\int_{\mathbb{R}^d}Z(t,\eps 
z)\psi(z) dz,
\]
where we made the change of variables \(y\leftrightarrow \eps z\). Take now
\[
R_0=\min\{t^{\alpha/\beta},\frac12\},
\]
whence for \(|z|\le R_0\) there holds
\[
R:=\frac{|z|^\beta}{2^\beta t^\alpha}\le \frac{1}{2^\beta}\quad\text{and}\quad \psi(z)\ge\mathrm{e}^{-\frac43}c_d.
\]
Note that \(R\le \frac{1}{2^\beta}\) implies \(R_\eps:=\frac{|\eps z|^\beta}{2^\beta t^\alpha}\le \frac{1}{2^\beta}\) 
for all 
\(\eps\le 1\). Hence the estimate
\[
Z(t,x)\sim t^{-\alpha}|x|^{\beta-d},\quad R\le\frac{1}{2^\beta}
\]
given by Lemma~\ref{fund_sol_asympt} allows us to estimate \(u^\epsilon\) from below as
\[
u^\eps(t,0)\gtrsim\int_{|z|\le R_0} Z(t,\eps z) dz\gtrsim \eps^{\beta-d}t^{-\alpha}\int_{|z|\le R_0} |z|^{\beta-d} dz.
\]
Since \(z\mapsto |z|^{\beta-d}\) is locally integrable,
\[
u^\eps(t,0)\gtrsim \eps^{\beta-d}\to\infty,\quad\eps\to 0.
\]

\noindent {\bf The lower bound. Case: \(d=\beta\).} We proceed as above. Again, Lemma~\ref{fund_sol_asympt} gives the estimate
\[
Z(t,z)\sim t^{-\alpha}(1+|\log(R)|)\sim -t^{-\alpha}\log(R),\quad R\le \frac{1}{2^\beta},
\]
which implies
\[
u^\eps (t,0)\gtrsim\int_{|z|\le R_0} Z(t,\eps z)d z\gtrsim t^{-\alpha}\int_{|z|\le 
R_0}\left|\log\left(\frac{|\eps z|^\beta}{2^\beta t^\alpha}\right)\right|d z.
\]
Substituting \(y=\frac{\eps z}{2t^{\alpha/\beta}}\) in the last integral gives the lower bound
\[
u^\eps(t,0)\gtrsim \eps^{-d}\int_{|y|\le\frac{\eps 
R_0}{2t^{\alpha/\beta}}}|\log(|y|)|d y.
\]
But the integral on the right hand side behaves like
\[
\eps^d|\log\eps|\quad\textrm{as $\eps\to 0$},
\]
whence
\[
u^\eps(t,0)\to\infty\quad\textrm{as $\eps\to 0$}.
\]

\noindent {\bf The upper bound.} We consider a point $x_0 \in \partial 
B_1(0)$. We have
\begin{align*}
u^\eps(t,x_0) &= \int_{\Rn} Z(t, x_0 -y) u_0^\eps(y) \, dy \\
& = \eps^{-d}\int_{B_\eps(0)} Z(t, x_0 -y)\psi\big(\tfrac{y}{\eps}\big) \, dy \\
& =\int_{B_1(0)} Z(t, x_0 -\eps z)\psi(z) \, dz.
\end{align*}

For $0< \eps \le \tfrac12$, we have that $x_0 -\eps z \in B_{1/2}^c(0)$ for $z \in 
\partial B_1(0)$. Since $Z(t,\cdot)$ is continuous and bounded in 
$B_{1/2}^c(0)$ uniformly with respect to \(\eps\) and \(\psi\) is integrable, we have by the Lebesgue dominated convergence 
theorem that 
\begin{equation*}
\lim_{\eps \to 0} u^\eps(t,x_0) =\int_{B_1(0)} \lim_{\eps \to 0} Z(t, x_0 -\eps z)\psi(z) 
\, dz = Z(t, x_0) < \infty,
\end{equation*}
for all $t>0$. Since for all $t>0$ we have that $u^\eps(t,0) \to \infty$ as $\eps \to 0$, the Harnack inequality cannot hold even after an arbitrarily long time lag. This finishes the proof.
\end{proof}

\begin{proof}[Proof of Theorem~\ref{counter_example_boundedness}]
Here we do not need the exact form of the fundamental solution in terms of a quite complicated $H$-functions. Instead, working in the frequency domain we can use the simpler Mittag-Leffler function \(E_\alpha\).

Let \(u_0(x)=\mathrm{e}^{-|x|^2}\) and \(u_0^n(x)=n^{d/p}u_0(nx)\). Then by~\cite[Theorem 2.12]{KempSiljZach16} the function
\[
u_n(t,x)=\left(Z(t,\cdot)\star u_0^n\right)(x)=\int_{\Rn}Z(t,x-y)u_0^n(y)\,dy
\]
is a classical solution of~\eqref{equation}. By using the Fourier transform and~\eqref{fund_fourier_trans} we may represent \(u_n(t,0)\) in the form
\[
\begin{split}
u_n(t,0)\,=& (2\pi)^{-d/2}\mathcal{F}^{-1}\left(E_\alpha(-|\xi|^\beta t^\alpha)\widehat{u_0^n}(\xi)\right)(0)\\
=&(2\pi)^{-d} n^{\frac{d}{p}-d}\int_{\Rn} E_\alpha(-|\xi|^\beta t^\alpha)\widehat{u_0}(\xi/n)\,d\xi\\
=&(2\pi)^{-d} n^{\frac{d}{p}}\int_{\Rn} E_\alpha(-|n\xi|^\beta t^\alpha)\widehat{u_0}(\xi)\,d\xi.
\end{split}
\]
By~\eqref{ML2_estimate} we have a lower bound
\[
E_\alpha(-\rho)\ge \frac{c_1}{1+\rho},\quad \rho\ge 0,
\]
for some positive constant \(c_1\). Moreover, since the Fourier transform of the Gaussian is also Gaussian, we have \(\widehat{u_0}>0\) in \(B(0,2\delta)\) with \(\delta=t^{-\alpha/\beta}\) and fixed \(t>0\). Hence for some positive constant \(c=c(t,\alpha,\beta,d)\) there holds
\[
u_n(t,0)\ge cn^{\frac{d}{p}}\int_{B(0,2\delta)}\frac{d\xi}{1+|n\xi|^\beta t^\alpha}\gtrsim n^{\frac{d}{p}-\beta}\int_{\frac{\delta}{n}}^{2\delta} r^{d-1-\beta} dr,
\]
where in the last step we introduced the spherical coordinates. This implies
\[
u_n(t,0)\gtrsim \begin{cases}
n^{\frac{d}{p}-\beta}, &\text{if $d>\beta$},\\
n^{\frac{d}{p}-\beta}\log(n), &\text{if $d=\beta$}.
\end{cases}
\]
In particular, \(u_n(t,0)\to \infty\), if (i) \(p<\frac{d}{\beta}\) and if (ii) \(p=1\) and \(d=\beta\). On the other hand, Young's inequality for convolutions~\cite[Theorem 1.2.10]{Graf04} and the fact that \(Z(t,\cdot)\) is a probability distribution function for all \(t>0\) (see~\cite{KempSiljZach16}) imply
\[
\|u_n(t,\cdot)\|_{L^p(\Rn)}\le\|Z(t,\cdot)\|_{L^1(\Rn)}\|u_0^n\|_{L^p(\Rn)}=\|u_0\|_{L^p(\Rn)}<\infty,
\]
which finishes the proof.
\end{proof}

\section{Harnack inequality}

\begin{proof}[Proof of Theorem~\ref{main_theorem}]

For the Harnack inequality, we need to consider the value 
of the solution in two separate points, and for this purpose, we denote
\begin{equation*}
R_1:=R_1(y) = \frac{|x_1-y|^\beta}{2^\beta t_1^\alpha} \quad \text{and}\quad R_2:=R_2(y) = 
\frac{|x_2-y|^\beta}{2^\beta t_2^\alpha}.
\end{equation*}

We split the integral defining the mild solution into two parts and have

\begin{align}\label{integral_split}
u(t_1,x_1) &= \int_{\Rn} Z(t_1, x_1-y) u_0(y) \, dy \notag\\
&= \int_{\{R_1 > 1 \}} Z(t_1, x_1-y) u_0(y) \, dy +  \int_{\{R_1 \le 1\}} Z(t_1, x_1-y) 
u_0(y) \, dy     \notag\\
& =: I_1 + I_2.
\end{align}

We will use the asymptotics of \(Z\) given in Lemma~\ref{fund_sol_asympt}. We start with the integral \(I_1\) and consider only the case \(\beta=2\), since this case is more difficult. The case \(\beta\neq 2\) follows from the argument for \(\beta=2\), since for \(\beta\neq 2\) there is no exponential function in the asymptotics for large values of the similarity variable \(R_1\) and we only need the estimate for the ratio
\[
\frac{|x_1-y|}{|x_2-y|}
\]
given below.

By Lemma~\ref{fund_sol_asympt} we have
\begin{align*}
Z(t, x)
\sim t^{-\frac{\alpha d}2}R^{\frac{d(\alpha-1)}{2(2-\alpha)}}e^{-\sigma 
R^{\frac1{2-\alpha}}},\quad R\ge 1,
\end{align*}
for a constant $\sigma=\alpha^{\frac\alpha{2-\alpha}}$. Observe that for
\begin{align*}
x_1, x_2 \in B_{r}(0),\quad  R_1 \ge 1 \quad\text{and}\quad t_1 \ge (2r)^{2/\alpha}
\end{align*}
there holds
\begin{align*}
R_2 = \frac{|x_2-y|^2}{4t_2^\alpha} &\ge \frac{(|x_1-y|-|x_1-x_2|)^2}{4t_2^\alpha} \\
&\ge \frac{(2t_1^{\alpha/2}-2r)^2}{4t_2^\alpha} \\
&\ge \frac14\left(\frac{t_1}{t_2}\right)^\alpha,
\end{align*}
which implies
\begin{align}\label{inclusion}
\{R_1 \ge 1 \}  \subset \left\{R_2 \ge \frac14\left(\frac{t_1}{t_2}\right)^\alpha \right\}.
\end{align}

 We obtain
\begin{align*}
I_1 &= \int_{\{R_1 > 1 \}} Z(t_1, x_1-y) u_0(y) \, dy \\
&\lesssim  t_1^{-\frac{\alpha d}2}\int_{{\{R_1 > 1\}}} 
R_1^{\frac{d(\alpha-1)}{2(2-\alpha)}}e^{-\sigma R_1^{\frac1{2-\alpha}}} u_0(y) \, dy \\
&=  t_1^{-\frac{\alpha d}2}\int_{\{R_1 > 1 \}} 
\left(\frac{R_1}{R_2}\right)^{\frac{d(\alpha-1)}{2(2-\alpha)}}e^{\sigma 
\left(R_2^{\frac1{2-\alpha}}- 
R_1^{\frac1{2-\alpha}}\right)}R_2^{\frac{d(\alpha-1)}{2(2-\alpha)}}e^{-\sigma 
R_2^{\frac1{2-\alpha}}} u_0(y) \, dy.
\end{align*}
Next we estimate the exponential factor from above. If $R_2 \le R_1$, there is nothing to  
prove, and for $R_2 \ge R_1$ we obtain for $\alpha \in (0,1]$ that 
\begin{align*}
R_2^{\frac1{2-\alpha}}- R_1^{\frac1{2-\alpha}} &\le (R_2-R_1)^{\frac1{2-\alpha}} \\
&= \left(\frac{|x_2-y|^2}{4t_2^\alpha} -\frac{|x_1-y|^2}{4t_1^\alpha} 
\right)^{\frac1{2-\alpha}} \\
&= \left(\frac{t_1^\alpha|x_2-y|^2-t_2^\alpha|x_1-y|^2}{4t_2^\alpha t_1^\alpha} 
\right)^{\frac1{2-\alpha}} \\
&= \left(\frac{(t_1^\alpha-t_2^\alpha)|y|^2-2y \cdot (t_1^\alpha x_2-t_2^\alpha 
x_1)+t_1^\alpha|x_2|^2-t_2^\alpha|x_1|^2}{4t_2^\alpha t_1^\alpha} 
\right)^{\frac1{2-\alpha}} \\
\end{align*}


Since $ (2r)^{2/\alpha} \le t_1 < t_2 \le t_1 + (2r)^{2/\alpha}\le 2t_1 $ and 
\(x_1,x_2\in B_r(0)\), we can estimate
\[
\frac{t_1^\alpha|x_2|^2-t_2^2|x_1|^2}{4t_2^\alpha 
t_1^\alpha}\le\frac{r^2}{4t_2^\alpha}\le\frac{1}{16}.
\]
It remains to estimate the sum of the first two terms in the brackets. We estimate the 
numerator of the second term as
\[
|y\cdot (t_1^\alpha x_2-t_2^\alpha x_1)|\le |y|(1+2^\alpha)r t_1^\alpha\le 3|y|r 
t_1^\alpha. 
\]
Denote \(r=|y|\) and study the parabola
\[
r\mapsto-a r^2+br
\]
with
\[
a=\frac{t_2^\alpha-t_1^\alpha}{4t_1^\alpha t_2^\alpha}\quad\text{and}\quad b=\frac{6r 
}{4 t_2^\alpha}.
\]
Since the maximum is attained at \(r=\frac{b}{2a}\), the upper bound 
for the sum of the first two terms is
\begin{equation}\label{est_blow_up}
\frac{b^2}{4a}=\frac{36 t_1^\alpha r^2}{16 t_2^\alpha 
(t_2^\alpha-t_1^\alpha)}\le\frac{3r^2}{t_2^\alpha-t_1^\alpha}\le\frac{3 
t_1^\alpha}{4(t_2^\alpha-t_1^\alpha)}=\left(\left(\frac{t_2}{t_1}\right)^\alpha-1\right)^{-1}.
\end{equation}

Next we estimate the power term. It is enough to estimate the ratio
\[
\frac{R_1}{R_2}=\Big(\frac{t_2}{t_1}
\Big)^\alpha\frac{|x_1-y|^2}{|x_2-y|^2}.
\]
Since the exponent is negative, we need a lower bound. In the set \(\{R_1\ge 1\}\) there 
holds
\begin{equation}\label{triangle_ineq_est}
|x_2-y|\le |x_1-x_2|+|x_1-y|\le 2r+|x_1-y|\le t_1^{\alpha/2}+|x_1-y|\le\frac32|x_1-y|,
\end{equation}
which gives the desired lower bound.

Altogether we obtain
\begin{align}\label{I1_est1}
I_1 &\le   Ct_2^{-\frac{\alpha d}2}\int_{\{R_2 \ge \frac1{16}\}}  
R_2^{\frac{d(\alpha-1)}{2(2-\alpha)}}e^{-\sigma R_2^{\frac1{2-\alpha}}} u_0(y) \, dy\notag \\
&\le  C\int_{\{R_2 \ge \frac1{16}\}} Z(t_2, x_2-y)u_0(y) \, dy,
\end{align}
where we used the inclusion~\eqref{inclusion} and estimated
\[
\Big(\frac{t_1}{t_2}\Big)^\alpha\ge \frac{t_1}{t_2}\ge\frac12.
\]

The proof shows that $C$ depends only on $d, 
\alpha$ and $t_2/t_1$. Moreover, we see from~\eqref{est_blow_up} that the constant \(C\) given by the proof 
blows up as \(t_2\to t_1\).

For $I_2$ we divide the treatment into two separate cases depending on the dimension. \\

\noindent {\bf Case: $d=\beta$.}

In this case we have from Lemma~\ref{fund_sol_asympt}
\begin{align*}
Z(t_1, x_1-y) \sim t_1^{-\alpha}\left(1-\log\left(\frac{|x_1-y|^\beta}{2^\beta t_1^\alpha}\right) 
\right),\quad R_1\le 1
\end{align*}
Noting that \(c(d,\beta)=-(d\pi)^{d/2}\) in~\eqref{riesz_kernel} we can estimate
\begin{align}\label{I2_est1}
I_2 &= \int_{\{R_1 \le 1 \}} Z(t_1, x_1-y) u_0(y) \, dy \notag\\
&\le Ct_1^{-\alpha} \int_{\{R_1 \le 1 \}} \left[1- 
\log\left(\frac{|x_1-y|^\beta}{2^\beta t_1^\alpha}\right) \right]u_0(y) \, dy\notag \\
&= Ct_1^{-\alpha} \int_{\Rn} \left[1+\beta\log 2+\alpha\log t_1-\beta\log|x_1-y| 
\right]\chi_{B_{r_1}(x_1)}u_0(y) \, dy \notag\\
&= Ct_1^{-\alpha}[(1+\beta\log 2+\alpha \log t_1+\beta c(d,\beta)G) \star 
\big(\chi_{B_{r_1}(x_1)}u_0\big)](x_1)\notag\\
&\le C\frac{[(c+\alpha \log t_1+\beta c(d,\beta)G) \star \big(\chi_{B_{r_1}(x_1)}u_0 
\big)](x_1)}{[(c+\alpha \log t_2+\beta c(d,\beta) G) \star  
\big(\chi_{B_{r_2}(x_2)}u_0\big)](x_2)}\int_{\{R_2 \le 1 \}} Z(t_2, x_2-y) u_0(y) \, dy.
\end{align}
In the last step we controlled the time factor $\left(\tfrac{t_2}{t_1}\right)^\alpha$
similarly as before and denoted \(c=1+\beta\log 2\). Combining~\eqref{integral_split},~\eqref{I1_est1} and~\eqref{I2_est1} gives the claim.\\

\noindent {\bf Case: $d>\beta$.}

Now we obtain
\begin{align*}
I_2 &= \int_{\{R_1 \le 1 \}} Z(t_1, x_1-y) u_0(y) \, dy \\
&\le Ct_1^{-\alpha} \int_{\{R_1 \le 1 \}} |x_1-y|^{\beta-d}u_0(y) \, dy \\
&= C\left(\frac{t_2}{t_1}\right)^\alpha\frac{\int_{\{R_1 \le 1 \}}  |x_1-y|^{\beta-d}u_0(y) 
\, dy}{\int_{\{R_2 \le 1 \}} |x_2-y|^{\beta-d}u_0(y) \, dy} \int_{\{R_2 \le 1 \}} Z(t_2, 
x_2-y)u_0(y) \, dy,
\end{align*}
where 
\begin{align*}
t_i^{-\alpha}\int_{\{R_i \le 1 \}} |x_i-y|^{\beta-d}u_0(y) \, dy &= \int_{\Rn} 
|x_i-y|^{\beta-d}\chi_{\{B_{r_i}(x_i)\}}u_0(y) \, dy \\
&= c(d,\beta)G*(u_0\chi_{\{B_{r_i}(x_i)\}})
\end{align*}
for $r_i = 2t_i^{\alpha/\beta}, i=1,2$. We obtain
\begin{align*}
I_2 \le C\left[\frac{[G \star (u_0\chi_{B_{r_1}(x_1)})](x_1)}{[G \star 
(u_0\chi_{B_{r_2}(x_2)})](x_2)}\right] \int_{\{R_2 \le 1 \}} Z(t_2, x_2-y)u_0(y) \, dy.
\end{align*}

Combining the estimates for \(I_1\) and \(I_2\) we have obtained
\begin{align*}
u(t_1, x_1) \le C\left[1+\frac{[G \star (u_0\chi_{B_{r_1}(x_1)})](x_1)}{[G \star 
(u_0\chi_{B_{r_2}(x_2)})](x_2)}\right] u(t_2, x_2),\quad d>\beta,
\end{align*}
for a constant $C$ depending on $d, \alpha,\beta$ and $t_2/t_1$.
\end{proof}

The Corollary~\ref{main_corollary} is now an easy consequence of the above proof.

\begin{proof}[Proof of Corollary~\ref{main_corollary}]

Since $r_2 \ge r_1$, by using the 
assumption~\eqref{initial_harnack} we obtain
\begin{align*}
G \star (u_0\chi_{B_{r_2}(x_2)})(x_2) 
&=\int_{\Rn} G(y)u_0(x_2-y)\chi_{B_{r_2}(x_2)}(x_2-y) \, dy \\
&=\int_{B_{r_2}(0)} G(y)u_0(x_2-y) \, dy \\
&\ge \inf_{B_{r_2}(x_2)}u_0\int_{B_{r_2}(0)} G(y) \, dy  \\
&\ge \frac1{H_0}\sup_{B_{r_1}(x_1)}u_0\int_{B_{r_1}(0)} G(y) \, dy \\
&\ge \frac1{H_0}\int_{B_{r_1}(0)} G(y)u_0(x_1-y) \, dy \\
&= \frac1{H_0}\int_{\Rn} G(y)(u_0\chi_{B_{r_1}(x_1)})(x_1-y) \, dy\\
&=G \star (u_0\chi_{B_{r_1}(x_1)})(x_1).
\end{align*}
Now the result follows from Theorem~\ref{main_theorem}.

\end{proof}

Finally, we turn into studying the one-dimensional case. 

\begin{proof}[Proof of Theorem~\ref{harnack_local_in_time}]
%
By the definition of the mild solution we have
\[
{u}(t,x)=\int_{\R^d}Z(t,x-y){u}_0(y)\, d y
\]
for nonnegative \({u}_0$.

We evaluate \({u}\) 
at the point \((t_1,x_1)\) and divide the integration into two 
parts as before depending on whether \(R_1=\frac{|x_1-y|^\beta}{2^\beta t_1^\alpha}\le 1\) or 
\(R_1\ge1\). The argument in the case \(R_1\le 1\) is the same for all \(\beta\), but in the case \(R_1\ge 1\) the argument is different depending on whether \(\beta=2\) or not.  We give here the proof in the case \(\beta\neq 2\). The case \(\beta=2\) can be handled similarly as in the proof of Theorem~\ref{main_theorem}. 

By Lemma~\ref{fund_sol_asympt}
\[
Z(t_1,x_1-y)\sim t_1^{-\frac{\alpha d}{\beta}},\quad R_1\le 1,
\]
and
\[
Z(t_1,x_1-y)\sim 
t_1^{\alpha}|x_1-y|^{-d-\beta},\quad 
R_1\ge 1.
\]
Write
\begin{align}\label{integral_split2}
u(t_1,x_1) &= \int_{\Rn} Z(t_1, x_1-y) u_0(y) \, dy\notag \\
&= \int_{\{R_1 > 1 \}} Z(t_1, x_1-y) u_0(y) \, dy +  \int_{\{R_1 \le 1\}} Z(t_1, x_1-y) 
u_0(y) \, dy     \notag\\
& =: I_1 + I_2
\end{align}
as in the proof of Theorem~\ref{main_theorem}.

\noindent {\bf The estimate for $I_1$.} The argument is essentially the same as in the proof of Theorem~\ref{main_theorem} for the power term. We will repeat the argument here with more details. Recall
\[
(2r)^{\alpha/\beta}\le t_1<t_2\le t_1+(2r)^{\beta/\alpha},
\]
which implies
\[
t_2\le t_1+(2r)^{\beta/\alpha}\le 2t_1.
\]
Then
\[
R_2 = \frac{|x_2-y|^\beta}{2^\beta t_2^\alpha} \ge \frac{(|x_1-y|-|x_1-x_2|)^\beta}{2^\beta t_2^\alpha} \ge \frac{(2t_1^{\alpha/\beta}-2r)^\beta}{2^\beta t_2^\alpha} \ge \left(\frac{t_1}{t_2}\right)^\alpha\ge\frac12,
\]
so
\[
\{R_1\ge 1\}\subset \{R_2\ge\tfrac12\}.
\]
Hence
\begin{align*}
I_1&\lesssim t_1^\alpha\int_{\{R_1\ge 1\}}|x_1-y|^{-d-\beta}u_0(y) dy\\
&\lesssim\left(\frac{t_1}{t_2}\right)^\alpha\int_{\{R_2\ge\tfrac12\}}\left(\frac{|x_2-y|}{|x_1-y|}\right)^{d+\beta} t_2^\alpha |x_2-y|^{-d-\beta}u_0(y) dy.
\end{align*}
Since the asymptotics of \(Z\) is sharp, it remains to have a uniform bound for the ratio
\[
\frac{|x_2-y|}{|x_1-y|}.
\]
But noting that \(t_1^{\alpha/\beta}\le\frac{|x_1-y|}{2}\) in the set \(\{R_1\ge 1\}\), we may estimate similarly as in~\eqref{triangle_ineq_est} to get
\[
|x_2-y|\le\frac32|x_1-y|,
\]
which gives the desired upper bound
\begin{equation}\label{I1_est2}
I_1\lesssim\int_{\{R_2\ge\tfrac12\}}Z(t_2,x_2-y)u_0(y) dy\le u(t_2,x_2).
\end{equation}

\noindent {\bf The estimate for $I_2$.} For \(R_1\le 1\) we have
\begin{align*}
|x_2-y|\le |x_1-y|+|x_1-x_2|\le 2t_2^{\alpha/\beta}+2r\le 3 t_2^{\alpha/\beta}\Rightarrow 
R_2\le\left(\frac{3}{2}\right)^\beta
\end{align*}
and we may  estimate
\begin{equation}\label{I2_est2}
\begin{split}
I_2&=\int_{\{R_1\le 1\}} 
Z(t_1,x_1-y)u_0(y)d y\\
&\lesssim 
t_1^{-\frac{\alpha d}{\beta}}\int_{\{R_1\le 1\}}u_0(y)d 
y\\
&\lesssim\left(\frac{t_2}{t_1}\right)^\frac{\alpha d}{\beta}\int_{\left\{R_2\le\left(\frac32\right)^\beta\right\}}Z(t_2,
x_2-y) u_0(y)d y\\
&\lesssim {u}(t_2,x_2).
\end{split}
\end{equation}
Combining~\eqref{integral_split2},~\eqref{I1_est2} and~\eqref{I2_est2} completes the proof.
\end{proof}

\section*{Acknowledgements}

J. S. was supported by the Academy of Finland grant 259363 and a V\"{a}is\"{a}l\"{a} foundation travel grant.



$\mbox{}$

\noindent {\footnotesize {\bf Dominik Dier}, Institute of Applied Analysis, University of Ulm, 89069 Ulm, Germany, e-mail:dominikdi3r@gmail.com.

$\mbox{}$

\noindent  {\bf Jukka Kemppainen}, Mathematics Division,
P.O. Box 4500,
90014 University of Oulu,
Finland
e-mail: jukemppa@paju.oulu.fi

$\mbox{}$

\noindent{\bf Juhana Siljander}, Department of Mathematics and Statistics, University of Jyv\"askyl\"a, P.O. Box 35,
40014 Jyv\"askyl\"a, Finland, e-mail: juhana.siljander@jyu.fi.

$\mbox{}$

\noindent {\bf Rico Zacher}, Institute of Applied Analysis, University of Ulm, 89069 Ulm, Germany, e-mail: rico.zacher@uni-ulm.de.

}

\end{document}